\documentclass[11pt,a4paper]{article}

\usepackage{epsf,epsfig,amsfonts,amsgen,amsmath,amstext,amsbsy,amsopn,amsthm,cases,listings,color
}
\usepackage{ebezier,eepic}
\usepackage{color}
\usepackage{multirow}
\usepackage{epstopdf}
\usepackage{graphicx}
\usepackage{pgf,tikz}
\usepackage{mathrsfs}
\usepackage[marginal]{footmisc}
\usepackage{enumitem}
\usepackage[titletoc]{appendix}
\usepackage{booktabs}
\usepackage{url}
\usepackage{subfigure}
\usepackage{mathrsfs}
\usetikzlibrary{arrows}

\allowdisplaybreaks[1]

\definecolor{uuuuuu}{rgb}{0.27,0.27,0.27}
\definecolor{sqsqsq}{rgb}{0.1255,0.1255,0.1255}

\newtheorem{definition}{Definition} 
\newtheorem{theorem}[definition]{Theorem}
\newtheorem{lemma}[definition]{Lemma}
\newtheorem{proposition}[definition]{Proposition}

\newtheorem{conjecture}[definition]{Conjecture}

\newtheorem{problem}[definition]{Problem}

\newenvironment{pf}{\noindent{\bf Proof}}

\def\qed{\hfill \rule{4pt}{7pt}}

\begin{document}
	\title{\bf\Large  Coloring hypergraphs that are the union of nearly disjoint cliques}
	
	\date{\today}
	
	\author{
		Dhruv Mubayi
		\thanks{Department of Mathematics, Statistics, and Computer Science, University of Illinois, Chicago, IL, 60607 USA.
			email: mubayi@uic.edu.
			Research partially supported by NSF awards DMS-1763317, DMS-1952767, DMS-2153576, a Humboldt Research Award and a Simons Fellowship.}
	\and Jacques Verstraete \thanks{Department of Mathematics, University of California, San Diego, CA, 92093-0112 USA.
			email: jverstraete@ucsd.edu.
			Research supported by NSF award DMS-1952786.}}
	\maketitle
	\begin{abstract}
	We consider the maximum chromatic number of hypergraphs consisting of cliques that have pairwise small intersections. Designs of the appropriate parameters produce optimal constructions, but these are known to exist only when the number of cliques is exponential in the clique size. We construct near designs where the number of cliques is polynomial in the clique size, and show that they have large chromatic number. 
	
	The case when the cliques have pairwise intersections of size at most one seems particularly challenging. Here we give lower bounds by analyzing a random greedy hypergraph process. We also consider the related question of determining the maximum number of caps in a finite projective/affine plane and obtain nontrivial upper and lower bounds. 
	\end{abstract}
	\section{Introduction}
For $1 \le \ell <  k \le q$, an $\ell$-$(q,k)$-system is a $k$-uniform hypergraph (henceforth $k$-graph) whose edge set is the union of cliques with $q$ vertices  that pairwise share at most $\ell$ vertices. Such hypergraphs are ubiquitous in combinatorics. Here are some examples:

\begin{itemize}
   
\item $\ell$-$(q,k)$-systems are extremal examples for many well-studied questions in extremal set theory, for example, an old open conjecture of Erd\H os~\cite{EC4} states that the maximum number of triples in an $n$-vertex triple system with no two disjoint pairs of edges with the same union is achieved by 1-$(5,2)$-systems. 
	
	\item Recently, Liu, the first author and Reiher~\cite{LMR} constructed the first family of hypergraphs that fail to have the stability property and $\ell$-$(q,k)$-systems were a crucial ingredient in constructing the extremal examples. 

\item Classical old open questions in projective geometry ask for the maximum size of caps in various  projective spaces. The triple system of collinear triples in projective space is a $1$-$(q,3)$-system by letting the $q$-sets be lines. Hence results about the independence number and chromatic number of $1$-$(q,3)$-systems have  connections to questions about large caps in projective spaces, which is a fundamental problem in finite geometry.
\end{itemize}

The chromatic number $\chi(H)$ of a hypergraph $H$ is the minimum number of colors required to color the vertex set of $H$ so that no edge of $H$ is monochromatic. A fundamental question about hypergraphs, first systematically investigated in the seminal work of Erd\H os and Lov\'asz~\cite{EL}, is to determine the maximum chromatic number of a  hypergraph with a specified number of edges. In this paper, we consider this problem for $\ell$-$(q,k)$-systems. Call a clique with $q$ vertices a $q$-clique.

\begin{definition} 
Given integers $1\le \ell < k\le q$ and $e \ge 1$, let $f_{\ell}(e,q,k)$ be the maximum chromatic number of an $\ell$-$(q,k)$-system where the number of $q$-cliques is $e$.
\end{definition}

Perhaps the most natural and interesting case is $\ell=k-1$ so in this case  we use the simpler terminology ``$(q,k$)-system'' and write $f(e,q,k)= f_{k-1}(e,q,k)$. We are interested in $f_{\ell}(e,q,k)$ when $k$ is fixed and both $q$ and $e$ are large. Special cases of this function have been extensively studied in the past. For example, the celebrated Erd\H os-Faber-Lov\'asz conjecture~\cite{EFL, K, KKKMO}  for graphs, which states that the maximum chromatic number of a collection of $q$ almost disjoint $q$-cliques is $q$, is the statement $f(q,q,2)=q$. Another example when $k>2$ is fixed,  is $f_{\ell}(e,k,k)$, the  largest possible chromatic number of partial designs, a classical question first studied by Erd\H os and Lov\'asz~\cite{EL}, and subsequently by Ajtai et.al.~\cite{AKPSS} whose results were sharpened by R\" odl and \v{S}i\v{n}ajov\'a~\cite{RS} and many others.

It is well-known that for $k$ fixed, every $k$-graph with $m$ edges has chromatic number  $O(m^{1/k})$. 
Indeed, color the vertices randomly and independently with $O(m^{1/k})$ colors. The expected number of monochromatic edges is $O(m^{1-(k-1)/k}) = O(m^{1/k})$. Now assign new colors one by one to some vertex inside each monochromatic edge to get a proper coloring. Altogether we used at most $O(m^{1/k})$ colors and there are no monochromatic edges. 

 An $(e,q,k)$-system has exactly $e{q \choose k}$ edges, so for fixed $k$, the argument above yields
$$f(e,q, k)  = O(e^{1/k} q).$$
If there is an $n$-vertex $q$-graph $H$ such that every $k$-set of vertices lies in exactly one edge, then the chromatic number of the $(q,k)$-system comprising the $q$-cliques of $H$ is exactly $n/(k-1)$. The number of $q$-sets in $H$ is $e={n \choose k}/{q \choose k} = (n)_k/(q)_k$ and hence the chromatic number of this $(q,k)$-system is of order $e^{1/k}q$. This shows that $f(e,q,k) = \Theta(e^{1/k}q)$ for $e=(n)_k/(q)_k$. However, all known constructions of these designs require the number of vertices, $n$, and hence also $e$, to be exponential in $q$.

Our first result below is a similar lower bound (construction) for $f(e,q,k$) when $e$ is polynomial in $q$. The construction combines algebraic and probabilistic ideas, by taking a random restriction of $(q,2)$ systems obtained from an affine plane -- see also~\cite{AMMV}.

\begin{theorem} \label{thm:main}
	Fix $k \ge 2$. Suppose that $q>k$ is sufficiently large and $2^q > e >(50q)^k$. Then there exists an $(e,q,k)$-system with  chromatic number at least $\Omega(e^{1/k} q)$. Consequently, for this choice of parameters, $f(e,q,k) = \Theta(e^{1/k} q)$.
\end{theorem}
The case $k=2$ deserves further mention.
Hindman~\cite{H} was the first to observe that the determination of $f(q,q,2)$ is equivalent to determining the edge chromatic number of the dual hypergraph, and using this formulation Chang and Lawler~\cite{CL} gave the following nontrivial upper bound. 

\begin{proposition}{\rm (\cite{CL})} \label{prop} 
Suppose that $2<q<e<q^2$. Then $e/4<f(e,q,2) < 3e/2$. 
\end{proposition}
The lower bound in the proposition is proved as follows: let $p$ be a prime such that $e^{1/2}/2<p \le e^{1/2}$ and let $A(2,p)$ be the affine plane of order $p$. Form $H$ by enlarging each line of $A(2,p)$ by adding $q-p\ge q-e^{1/2} \ge 0$ new vertices such that distinct lines have disjoint enlargements. The resulting $(q,2)$-system $H$ has $p^2\le e$ $q$-cliques, and we may add disjoint $q$-sets arbitrarily so that we have exactly $e$ cliques. In a proper coloring, every two vertices in $A(2,p)$ must receive distinct colors, and hence $\chi(H) \ge p^2>e/4$. 

If we use the prime number theorem, then as  $q\rightarrow \infty$, this construction  yields $f(e,q,2)\ge (1+o(1))e$. Kahn~\cite{K} proved an upper bound for $f(e,q,2)$ that is asymptotically optimal in the range $e< q^2$ so  as $q \rightarrow \infty$, 
$$f(e,q,2) = (1+(1))e \qquad \hbox{ for } \qquad q \le e \le q^2.$$
There is a  further improvement $f(e,q,2) \le e$ for $e$ sufficiently large~\cite{KKKMO}.

The next case $k=3$ seems wide open in the case that $e$ is small. For $e>q^3$, we have $f(e,q,3)= \Theta(e^{1/3}q)$ by Theorem~\ref{thm:main}. For $q^{3/2}<e<q^3$, we have the lower bound $\Omega(e^{2/3})$ by taking the circles in an inversive plane of order $(1+o(1))e^{1/3}$ and adding disjoint sets of new points to the circles to create $q$-sets. Any three points in the inversive plane lie in a circle and hence an edge of our hypergraph, so the independence number is at most two, and the chromatic number is at least half the number of vertices in the inversive plane which is $\Theta(e^{2/3})$. Apart from this we have no nontrivial lower or upper bounds in the range $q<e<q^3$.

\begin{problem}
Determine the order of magnitude of $f(e,q,3)$ in the range $q<e<q^3$.
\end{problem}

Another  interesting case  is $f_1(e,q,k)$ when $k>2$ (the $q$-sets form a  linear $q$-graph). Here we prove the following theorem which gives bounds that get closer  as $k$ increases. The lower bound is obtained via a random greedy algorithm. It is one of the few instances where a constrained random $q$-graph process on $n$ vertices is analyzed where $q$ is polynomial in $n$.
\begin{theorem} \label{thm:main2}
	Fix $k \ge 3$. There exist a constant $C$  such that 
 $$e^{\frac{1}{2k-2}}q^{1-\frac{2}{k-1}+o(1)}  <
 f_1(e,q,k) < C\, e^{\frac{1}{2k-2}} q$$
 where the lower bound holds for $e>q^{2k+2}$.
\end{theorem}
We believe that the upper bound above is sharp in order of magnitude for $e$  sufficiently large.

\begin{conjecture}
Fix $k \ge 3$. There exists $C=C_k>0$ such that if $e>q^C$, then
$$f_1(e,q,k) = \Theta(e^{\frac{1}{2k-2}}q) \qquad (q \to \infty).$$
\end{conjecture}

It is an interesting wide open problem to study the behavior of $f_1(e,q,k)$ for smaller values of $e$.
The smallest case $k=3$ naturally gives rise to a particular class of 1-$(e,q,3)$ systems, namely the collection of collinear triples in lines in finite projective planes. In fact, one possible way to obtain 1-$(q,3)$ systems is to take projective planes with few caps, and then take a random restriction of the points. While it is not clear how to obtain constructions in this manner that supersede the random greedy algorithm, it does suggest another fundamental question, namely to count the number of caps in a finite projective/affine plane.

 Define $I_q$ to be the maximum number of caps in an affine plane of order $q$. In a Desarguesian plane the normal rational curve yields a cap of size at least $q+1$ and since all subsets of a cap are caps, we have at least $2^{q+1}$ caps. However, in general affine planes, the largest known caps that are guaranteed to exist are of order $(q \log q)^{1/2}$ giving $I_q > 2^{c\, (q \log q)^{1/2}}$.
 On the other hand, the maximum size of a cap in any affine plane of order $q$ has size at most $q+2$, and this yields the bound $I_q < {q^2 \choose q+2} < 2^{c\, q \log q}$. We give improvements to both these bounds.

 \begin{theorem} \label{maincaps}
 	There is an absolute constant $c$ such that for all sufficiently large $q$ for which $I_q$ is defined, $$2^{c\,q^{1/2} (\log q)^{3/2}} < I_q < 2^{6q}.$$
 	\end{theorem}
	
 	There is still a huge gap between our upper and lower bounds for $I_q$, however, as mentioned earlier, if we restrict to the case of Desaguesian/Galois planes, our upper bound is much closer to the truth. 
 Define $I'_q$ to be the maximum number of caps in a Galois affine plane $A(2,q)$. Thus $I'_q \le I_q$ and perhaps equality holds. Then Theorem~\ref{maincaps} yields $2^{q+1}< I'_q < 2^{6q}$. 

 \begin{conjecture}
 $I'_q = 2^{(1+o(1))q}$.
\end{conjecture}

\section{Random restrictions of Affine planes}
We will use the following version of the Chernoff bounds.

\begin{lemma} [Multiplicative Chernoff Bounds]
   Suppose $X_1, \ldots,  X_n$ are independent random variables taking values in $\{0, 1\}$. Let $X$ denote their sum and let $\mu = E[X]$ denote the  expected value of $X$. Then for any $\delta > 0$,
$$
\Pr(X\le(1-\delta)\mu) < \exp(-\delta^2\mu/2)$$
   $$
\Pr(X \ge (1+\delta)\mu) < \exp(-\delta^2\mu/(2+\delta)).$$
\end{lemma}

{\bf Proof of Theorem~\ref{thm:main}.} By the prime number theorem, there is a prime number $Q\ge q$ such that $(1/2)e < Q^k+Q \le e$. Let $V=\mathbb F_Q \times \mathbb F_Q$. Given a  polynomial $p(x)$  over $\mathbb F_Q$ let $S(p(x))=\{p(x): x \in \mathbb F_Q\}$. Let $H_Q$ be the $Q$-graph with vertex set $V$ and edge set $$\{S(p(x)): \deg(p(x)) < k)\} \cup \{C_x: x \in \mathbb F_Q\}$$
 where $C_x=\{(x,y):y \in \mathbb F_Q\}$ is the column of $x$. The number of edges in $H_Q$ is $Q^k +Q\le e$.

Let $W$ be a random subset of $V$ obtained by picking each element of $V$ independently with probability 
$p = \frac{q}{10Q}$. Given an edge $f \in H_Q$, the expected size of $f \cap W$ is $q/10$ and the multiplicative Chernoff bound with $\mu= q/10$ and $\delta=9$ implies that the probability that $|f \cap W|>q$ is at most
$\exp(-81q/110)$. The number of edges $f$ in $H_Q$ is  $Q^k+Q\le e<2^q$, so the union bound implies that
$$\Pr(\exists f: | f\cap W|>q)< 2^q\exp(-81q/110)<(0.96)^q.$$ For $C$ a sufficiently large constant, the probability that $|W|<pQ^2/C$ is by the multiplicative Chernoff bound (with $\mu =pQ^2$ and  $\delta=1-1/C$) at most 
$$\exp(-0.49pQ^2)=\exp(-0.049qQ).$$ Since $e>(50q)^k$, we have $Q>(e/4)^{1/k}\ge 25q$, and $\exp(-0.049qQ)<\exp(-q^2)$. Since $(0.96)^q+\exp(-q^2)<1$, there is a set $W$ such that $|f \cap W|\le q$ for all $f$ and $|W|=\Omega(qQ)$.
Let $H'$ be the $k$-graph with vertex set $W$ whose edges  are the $Q^k+Q$ cliques ${f \cap W \choose k}$. 

 We now prove that $\alpha(H')\le (k-1)^2+1$.
 
{\bf Claim.} Given $(x_1, y_1), \ldots, (x_k, y_k)$ in $\mathbb F_Q \times \mathbb F_Q$ with the $x_i$s distinct, there is a (unique) polynomial $p(x)$ of degree less than $k$ with $p(x_i)=y_i$ for $i \in [k]$.

{\bf Proof of Claim.} Write $p(x) = \sum_{j=0}^{k-1}a_jx^j$. Then the conclusion of the Claim is equivalent to the matrix equation $Ba=y$, where $B$ is the $k$ by $k$ Vandermonde matrix with parameters $x_1, \ldots, x_k$, $a=(a_0, \ldots, a_{k-1})^T$ and $y=(y_1, \ldots, y_{k})^T$. Since the $x_i$s are distinct, $B$ is invertible and hence there is a unique solution $a$.
\qed
 
 Pick a set $I\subset W$ of vertices of size at least $(k-1)^2+1$. If $I$ has at least $k$ vertices in some column $C_x$, then these $k$ vertices lie in $C_x \cap W$ and hence lie in an edge of $H'$. Therefore, by the pigeonhole principle,  $I$ has at least $k$ vertices in distinct columns. By the Claim, these $k$ vertices lie in a unique $S(p(x))$ and hence lie in an edge of $H'$. This proves that  $\alpha(H')\le (k-1)^2+1$. 
 
 Finally, we modify $H'$ by adding $q-|f \cap W|$ new vertices to each set $f \cap W$ so as to make a clique of size exactly $q$. This produces an $(e,q,k)$-system
with chromatic number at least $|W|/k^2 = \Omega(qQ)= \Omega(e^{1/k}q)$ as required.\qed

\section{The case $\ell =1$}

Here we prove the upper bound   in Theorem~\ref{thm:main2}.  Note that the number of $k$-sets in a $(q,k)$-system where the number of $q$-cliques is $e$ is $e{q \choose k}$ so the trivial bound is $f_1(e,q,k)=O(e^{1/k} q)$. 
\begin{theorem}
	For each fixed $k \ge 3$, we have
	$f_1(e,q,k)  = O(e^{1/(2k-2)}q)$.
	\end{theorem}
\begin{proof} Let $H$ be a $(q,k)$-system with $e$ edges. Put $d:=e^{1/2} q^{k-1}$. Let  $A=\{v \in V(H): d(v)\le d\}$ and $B = V(H)\setminus B$. Since $\Delta(H[A]) \le d$, the local lemma implies that there is a proper coloring of $H[A]$ with at most $O(d^{1/(k-1)}) = O(e^{1/(2k-2)}q)$ colors.
Since 
	$$ke{q \choose k} \ge \sum_{v \in B} d(v)  \ge |B| d = |B|e^{1/2} q^{k-1}$$ we obtain $|B|  = O(e^{1/2}q)$. Now consider $v \in B$. The edges in $H$ containing $v$ that lie within $B$ are all in subsets of $q$-cliques  containing $v$. Let $A_1, \ldots, A_p$ be the set of $q$-cliques containing $v$ that have at least $k$ vertices in $B$ and let $a_i=|A_i \cap B|$. Then the degree of $v$ in $H[B]$ is $\sum_i {a_i-1 \choose k-1}$.   Since $A_i \cap A_j = \{v\}$, we have $\sum (a_i-1) < |B| = O(e^{1/2}q)$. The quantity $\sum_i {a_i-1 \choose k-1}$ subject to this constraint is maximized when as many of the $a_i$ are as large as possible and the rest are as small as possible. Since $a_i \le q$, we obtain
	$$\sum_i {a_i-1 \choose k-1} \le \frac{|B|-1}{q-1} {q-1 \choose k-1} = O (e^{1/2}q^{k-1}) = O(d).$$
	Hence the maximum degree of $H[B]$ is  $O(d)$ and once again we can properly color $H[B]$ with $O(e^{1/(2k-2)}q)$ colors. 
 We always use colors that have not been used in $H[A]$. In particular, this implies that if there is a $k$-set that has vertices in both $A$ and $B$, then it will not be monochromatic in our coloring.  The resulting coloring is a proper coloring of $H$ with $O(e^{1/(2k-2)}q)$   colors. 
	\end{proof}

{\bf Remark.} A further improvement by polylog factors can be achieved above using more advanced results to color the vertices in $B$.

\section{Lower bound for 1-$(q,k)$ systems using random greedy}
In this section we prove the lower bound in Theorem~\ref{thm:main2}: $f(e,q,k) = \Omega^*(e^{1/(2k-2)}
			q^{1-2/(k-1)})$ for $e > q^{2k+2}$ where the $*$ indicates some extra polylog factors in $e$.

Fix $k \ge 3$. Consider the random greedy $(q,k)$-process: We pick a $q$-set $e_1$ of $[n]$ at random. Given that we have picked $e_1, \ldots, e_i$, we pick a $q$-set $e_{i+1}$ randomly (with equal probability) from all other $q$-sets that do not intersect any of $e_1, \ldots, e_i$ in more than one point. Eventually we obtain a (random) $q$-graph $G^q$ with $e$ edges, and also a random $k$-graph $H=H^k = \cup_{i=1}^e {e_i \choose k}$, which has $e{q \choose k}$ edges.

Fix a $t$-set $I$ of $[n]$ and let us calculate the probability that $I$ is an independent set in $H^k$. For $i \ge 0$, let 
$$W_j=\left\{ S \in {[n] \choose q}: |S \cap I|=j\right\}$$
and put $M = \cup _{j=k}^q W_j$ and $m:=|M|$. Here $M$ stands for {\it missing} $q$-sets since these $q$-sets cannot be present in $G^q$ as $I$ is an independent set in $H^k$. Note that
$w_j:= |W_j| \le {t \choose j} {n \choose q-j}$. Say that  $A \in W_j$ {\em blocks}  $B \in M$ if  $|A \cap B| \ge 2$ and let $b_j$ be the number of sets in $M$ blocked by an $A \in W_j$ (it is the same for all $A$). 
We write $f \gg g$ to mean that there is a (large) positive constant $c=c_k$ such that $f \geq cg$. From now on, we assume that
$$n \gg tq^2 \qquad \hbox{ and } \qquad t \gg 10k\left(\frac{n^{k-2}\log n
}{q^{k-4}}\right)^{\frac{1}{k-1}}.$$
We then have
$$b_0 < q^2 t^k {n \choose q-k-2}$$
$$b_1 \le q{t \choose k-1} {n \choose q-k-1} + q^2t^{k}{n \choose q-k-2} < qt^{k-1} {n \choose q-k-1}$$
and for $2\le j \le k-1$,
$$b_j \le {j \choose 2}{t \choose k-2}{n \choose q-k} + jqt^{k-1}{n \choose q-k-1} + q^2t^k{n \choose q-k-2} < 2t^{k-2}{n \choose q-k},$$
where we use $n/tq^2 \rightarrow \infty$ in the last two displays. Note that $m= \Theta(t^k{n \choose q-k})$. 
Let $a_j = m/b_j$ so that
$$a_0 = \Theta\left(\frac{n^2}{q^4}\right) \qquad a_1 = \Theta\left(\frac{tn}{q^2}\right) \qquad 
a_j = \Theta(t^2) \qquad \hbox{ for $2\le j \le k-1$}.$$

{\bf Claim.}  If $I$ is an independent set, then at least $a_j/k $ edges of $W_j$ must be present in $G^q$ for some $0 \le j \le k-1$. 

{\bf Proof.} If not,  since no set of $M$ is in $G^q$, the number of sets in $M$ that are blocked is less than $\sum_{j=0}^{k-1} b_j(a_j/k) \le  m$. This means that some edge of $M$ would be in $G^q$, contradicting the fact that $I$ is an independent set. 

 Let $G_i=G_i^q$ be the (random) $q$-graph obtained after $i$ edges have been added and let $e_i$ be the $q$-set added at step $i$. Define 
$$\ell := \min\{i: |W_j \cap G_i| \ge a_j/3k \hbox{ for some $j=0,1,\ldots, k-1$}\}.$$
	In words, $\ell$ is the smallest index such that $G_i$ contains at least $a_j/3k$ edges from $W_j$ for some $j$. 

For $0\le j \le k-1$, let $A_j$ be the event that $I$ is an independent set and $e_{\ell} \in W_j$. The Claim, the definition of $\ell$, and the union bound imply that 
$$\Pr(I \hbox{ is independent}) \le  \sum_{j=0}^{k-1} \Pr(A_j).$$ 
 Let $m_i=|M \setminus G_i|$ . Since $|W_j \cap G_{\ell}|\le a_j/3k$ for each $j \in \{0,1,\ldots, k-1\}$, we have for $i \le \ell$, 
 $$m_i \ge m_{\ell} \ge m- \sum_{j=0}^{k-1}b_j(a_j/3k) \ge m - km/3k > m/2.$$
  For each $S \subset \{1, \ldots, e\}$, define the event
$$A_0(S) := \{G^q \in A_0: e_i \in W_0  \Leftrightarrow i \in S\}.$$ 
These events are disjoint for distinct $S$ and hence
$$\Pr(A_0) = \sum_{S \subset [e]} \Pr(A_0(S)).$$
If $\Pr(A_0(S))>0$, then since $I$ is an independent set there is a subset $S' \subset S$ with $|S'|\ge a_0/3k$  and $m_i \ge m/2$ for all $i \in S'$. Hence we can further write
$$\Pr(A_0) = \sum_{|S| \ge a_0/3k} \Pr(A_0(S)).$$
Write  $w_{j,i}=|W_j \setminus G_i|$ and $r_i = w_{1,i}+\cdots + w_{k-1,i}$  for $j=0,1,\ldots, k-1$. Then
\begin{align*} \Pr(A_0) 
&\le \sum_{|S| \ge \frac{a_0}{3k}} \left(\prod_{i \in S} \Pr(e_i \in W_0) \prod_{i \not \in S} \Pr(e_i\in W_1 \cup \cdots \cup W_{k-1})\right) \\
&=\sum_{|S| \ge \frac{a_0}{3k}} \prod_{i \in S} \left(\frac{ w_{0,i}}{w_{0,i}+m_i+r_i}\right)  \prod_{i \not\in S}
\left(\frac{r_i}{w_{0,i} +m_i+r_i}\right)\\
& = \sum_{|S| \ge \frac{a_0}{3k}} \prod_{i \in S} \left(\frac{ w_{0,i}}{w_{0,i}+m_i}\right) 
\left(\frac{w_{0,i}+m_i}{w_{0,i}+ m_i+r_i}\right)\prod_{i \not\in S}
\left(\frac{r_i}{w_{0,i} +m_i+r_i}\right) \\
& = \sum_{|S| \ge \frac{a_0}{3k}} \prod_{i \in S} \left(1-\frac{ m_i}{w_{0,i}+m_i}\right) 
\left(\frac{w_{0,i}+m_i}{w_{0,i}+ m_i+r_i}\right)\prod_{i \not\in S}
\left(\frac{r_i}{w_{0,i} +m_i+r_i}\right) \\
&\le \exp\left(-\frac{m a_0}{6kw_0}\right)
	\sum_{|S| \ge \frac{a_0}{3k}} \prod_{i \in S}
	\left(\frac{w_{0,i}+m_i}{w_{0,i}+ m_i+r_i}\right)\prod_{i \not\in S}
	\left(\frac{r_i}{w_{0,i} +m_i+r_i}\right) \\
	& \le \exp\left(-\frac{m a_0}{6kw_0}\right)
	\sum_{S \subset [e]} \prod_{i \in S}
	\left(\frac{w_{0,i}+m_i}{w_{0,i}+ m_i+r_i}\right)\prod_{i \not\in S}
	\left(\frac{r_i}{w_{0,i} +m_i+r_i}\right) \\
	&= \exp\left(-\frac{m a_0}{6kw_0}\right)
\prod_{i =1}^e
	\left(\frac{w_{0,i}+m_i}{w_{0,i}+ m_i+r_i}+\frac{r_i}{w_{0,i} +m_i+r_i}\right) \\
	&=\exp\left(-\frac{m a_0}{6kw_0}\right).
\end{align*}

We observe that ${n \choose t} \Pr(A_0) < 1/k$ since this follows from
$$\frac{ma_0}{6kw_0} = \frac{m^2}{6kb_0w_0}
> \frac{[{t \choose k}{n \choose q-k}]^2}{6k q^2 t^k {n \choose q-k-2} {n \choose q}} > t \log n$$
which holds due to $t^{k-1} \gg 10k n^{k-2}\log n / q^{k-4}.$
Similarly, ${n \choose t} \Pr(A_1) < 1/k$ follows from 
$$\frac{m^2}{6kb_1w_1}
> \frac{[{t \choose k}{n \choose q-k}]^2}{6k q t^{k-1} {n \choose q-k-1} t {n \choose q-1}} \gg  t \log n$$
using the weaker bound $t^{k-1} \gg n^{k-2} \log n/q^{k-3}$. Finally,
${n \choose t} \Pr(A_j) < 1/j$ for each $2\le j \le k-1$ follows from 
$$\frac{m^2}{6kb_jw_j}
> \frac{[{t \choose k}{n \choose q-k}]^2}{12k t^{k-2+j} {n \choose q-k} {n \choose q-j}} \gg  t \log n.$$
This is equivalent to
$$t > \left(\frac{n^{k-j}\log n}{q^{k-j}}\right)^{\frac{1}{k+1-j}}.$$

	Using  $t \gg 10k\left(\frac{n^{k-2}\log n
	}{q^{k-4}}\right)^{\frac{1}{k-1}}$ this follows from 
$$\left(\frac{n^{k-2}\log n}{q^{k-4}}\right)^{k+1-j} \gg
	\left(\frac{n^{k-j}\log n}{q^{k-j}}\right)^{k-1}$$
		which is equivalent to
		$n^{j-2} \gg  (\log n)^{j-2} q^{3j-2k-4}$ and this is trivial using $n>q$ and $j<k$.

So with positive probability  $\alpha(H^k)<t$.

Let $e$ be the number of edges in $G^q$. Certainly $e < n^2/q^2$ since $G^q$ is a $(q,k)$-system and we may add edges arbitrarily so that $e=n^2/q^2$. Setting $n = C tq^2$ for large $C$ we obtain that 
$$\chi(H^k) \ge \frac{n}{t} =\Omega\left(\left(\frac{n q^{k-4}}{\log n}\right)^{\frac{1}{k-1}}\right) =
\Omega^*\left(
{e^{\frac{1}{2k-2}}
		q^{1-\frac{2}{k-1}}
		}
		\right).$$
	Consequently, $f(e,q,k)=
	\Omega^*(
	e^{1/(2k-2)}
			q^{1-2/(k-1)})$ and a short calculation solving for $t$ and $n$ in terms of $q$ yields that this holds for $e>q^{2k+2}$	\qed
		
		In the proof above we used the trivial bound $e <n^2/q^2$. We do not know that this is tight. The value of $e$ could theoretically be as low as $n^2/q^3$ in which case our lower bound on $f(e,q,k)$ would be $\Omega^*(e^{1/(2k-2)}q^{1-3/(2k-2)})$. It is a very interesting question to study the final number of edges in the $(q,k)$-process when $q$ polynomially related to $n$.  
		
		The random greedy process above requires $e$ large while our upper bound $f(e,q,k) = O(e^{1/(2k-2)}
			q)$ applies for all $e$, so it is interesting to study the situation for $e<q^{2k+2}$. Taking the collinear triples in a projective plane quickly yields 
		$$f(q^2+q+1, q, 3)\ge (q^2+q+1)/(q+2) >q-1 = \Omega(e^{1/4}q^{1/2})$$ for $e=\Theta(q^2)$.

\section{Counting caps in affine/projective spaces}

In this section we prove Theorem~\ref{maincaps}.  We require the following lemma from~\cite{CM}. Recall that $\tau(F)$ is the minimum size of a vertex subset of $F$ that intersects every edge of $F$.

 \begin{lemma}\label{hittinglemma}
 	Suppose $F$ is an $s$-uniform hypergraph,
 	and $z_i$, $i \in V(F)$ are independent random indicator variables with $\Pr[z_i = 1] = p$, 
 	for all $i \in V(F)$.
 	Let 
 	\[
 	F' = \{A \in F: \forall i \in A, z_i=1\}.
 	\]
 	Suppose there exists $\alpha > 0$ such that $|F|p^{s(1-\alpha)} < 1$.
 	Then for any $c \geq e2^ss\alpha$,
 	\[
 	\Pr[\tau(F') > s^2(c/\alpha)^{s+1}] \leq s^2 |V(F)|^{s-1}p^c.
 	\]
 \end{lemma}
 \bigskip
 
 The reason we need to use the lemma above is to take care of codegrees that may be of logarithmic size (the codegree of a pair of vertices is the number of edges containing them both). We also need the following result from~\cite{CDM}; the result as stated in~\cite{CDM} applies to linear hypergraphs but standard methods imply the same bound for hypergraphs with bounded maximum codegree. 

 \begin{theorem} [\cite{CDM}, Theorem 3] \label{cdm}
     Fix $s>0$. There exists a constant $c=c_s>0$ such that the number of independent sets in every $n$-vertex 3-graph with average degree $d$ and maximum codegree $s$ is at least $2^{c\,n(\log d)^{3/2}/d^{1/2}}$.
 \end{theorem}
 
{\bf Proof of Theorem~\ref{maincaps}.}
We start with the lower bound.
A cherry is the 5 vertex 3-graph comprising three edges, every two of which share the same two vertices. Let $H$ be a $(q,3)$-system with $n=q^2$ vertices, average degree at most $q^3$, and maximum codegree at most  $q$. Note that the collinear triples of any affine plane of order $q$ can be viewed as such a $(q,3)$-system $H$.

Our plan is to take a random induced subgraph of $H$ on $q^{3/5-\epsilon}$ vertices where there is a small set of vertices that touches every cherry. To this end we apply Lemma~\ref{hittinglemma} with $F$ being the 5-graph of copies of cherries in $H$ and $p=q^{-7/5-\epsilon}$. Letting $\alpha$ be sufficiently small in terms of $\epsilon$, we obtain
$$|F|p^{(1-\alpha)s}<n^2q^3p^{(1-\alpha)s}<q^7q^{-5(1-\alpha)(7/5+\epsilon)} = o(1).$$
Hence for $c$ a large constant, Lemma~\ref{hittinglemma} and standard Chernoff bounds yield that with probability greater than 0.9 say, a random induced subgraph $H'$ of $H$ with $p$ as above has $m=\Theta(q^{3/5-\epsilon})$ vertices, average degree  $d =O(p^2q^3)=O(q^{1/5-2\epsilon})$ and has a set $S$ of $O(1)$ vertices whose removal  makes it $F$-free. Since $H'-S$ has no cherry, it has maximum codegree at most two.  Theorem~\ref{cdm} now implies that the number of independent sets in $H'-S$, and hence also in $H$,  is at least $2^{c\, q^{1/2} (\log q)^{3/2}}$.

  We now prove the upper bound. Let $G=(P,L,E)$ be the point line incidence graph of the affine plane of order $q$.  It is well known that $G$ satisfies the bipartite expander mixing lemma:
 	$$\left|e(X,Y) - \frac{d}{n}|X||Y|\right|< \sqrt{q|X||Y|}$$
 	for all sets $X \subset P$ and $Y \subset L$. Let $I_{q,t}$ be the number of caps of size $t$ and let us enumerate such caps in the affine plane by choosing sequences of $t$ vertices in $P$. Note that $I_{q,t}=0$ for $t>q+2$. The first two vertices can be chosen arbitrarily. Given that $i \ge 2$ and vertices $v_1, \ldots, v_i$ have been chosen, let $Z_i=\{y \in L: \exists x, x' \in X, xy, x'y \in E\}$   be the set of common neighbors of pairs of vertices in $S_i=\{v_1, \ldots, v_i\}$.
 	Let $X_i \subset P$ be the set of nonneighbors of $Z_i$. The next vertex $v_{i+1}$ in the cap must come from $X_i$. The expander mixing lemma now yields
 	$$0 =e(X_i, Z_i)\ge \frac{|X_i||Z_i|}{q} - \sqrt{q|X_i||Z_i|}=\frac{|X_i|{i \choose 2}}{q} - \sqrt{q|X_i|{i \choose 2}}.$$
 	This yields $|X_i|\le q^3/{i \choose 2}.$ Consequently, For $2 < t \le q+2$, 

  $$I_{q,t} \le \frac{q^4\prod_{i=2}^{t-1} \min\{q^2, q^3/{i \choose 2}\}\}}{t!}.$$
If $t<2\sqrt q$ we just use the bound
$$I_{q,t} \le \frac{q^4\prod_{i=2}^{t-1} q^2}{t!} = 2^{O(\sqrt q \log q)}.$$
For $t>2\sqrt q$, using $n!>(n/e)^n$ we have
$$I_{q,t} \le \frac{
2^{O(\sqrt q \log q)}
\prod_{i=2\sqrt{q}}^{t-1} \frac{2q^3}{i(i-1)}
}
{t!} < 2^{O(\sqrt q \log q)} \frac{2^tq^{3t}}{(t!)^3}< 
2^{O(\sqrt q \log q)} \left(\frac{2^{1/3}e q}{t}\right)^{3t}
.$$
  An easy calculus exercise shows that the maximum of $(2^{1/3}e q/t)^{3t}$ for $2\sqrt q < t \le q+2$ and $q$ sufficiently large occurs either at the endpoint $t=q+2$ or at $t=2^{-2/3}e q$ (all logs are base 2). Checking both these points we see that the maximum is in the latter case and is less than $2^{5.77q}$. Consequently, $I_{q,t} < 2^{5.77q+ O(\sqrt q \log q)}$ and for $q$ sufficiently large,
$$I_q = q^2 + q^2(q^2-1) + q \cdot \max_{3\le t \le q+2}I_{q,t} <  2^{5.77q+ O(\sqrt q \log q)}< 2^{6q}$$
 as desired.
 	\qed

\end{document}